\documentclass[11pt,a4paper,leqno]{amsart}
\usepackage[T1]{fontenc}
\usepackage[utf8]{inputenc}
\usepackage{amsmath,amssymb,amsthm,mathtools}
\usepackage{lmodern}
\usepackage[expansion=false]{microtype}
\usepackage[hmargin=2.6cm,vmargin=2.5cm,footskip=1cm]{geometry}
\usepackage[hidelinks]{hyperref}
\usepackage{xcolor}

\newcommand{\D}{\mathbb D}
\newcommand{\E}{\mathbf E}
\renewcommand{\P}{\mathbf P}
\newcommand{\Var}{\operatorname{Var}}
\newcommand{\dd}{\,\mathrm d}
\newtheorem{theo}{Theorem}[section]
\newtheorem{lem}[theo]{Lemma}
\newtheorem{propo}[theo]{Proposition}
\newtheorem{coro}[theo]{Corollary}
\theoremstyle{definition}
\newtheorem{defin}[theo]{Definition}
\newtheorem{remark}[theo]{Remark}
\newtheorem*{remark*}{Remark}
\numberwithin{equation}{section}
\title[Wiman--Valiron inequalities in the unit disk]{Wiman--Valiron inequalities in the unit disk\\
outside sets of finite logarithmic measure}
\author{V\'ictor J. Maci\'a}
\address{Departamento de An\'alisis Matem\'atico, Universidad de La Laguna, Spain}
\email{victor.macia@ull.edu.es}
\thanks{Departamento de An\'alisis Matem\'atico, Universidad de La Laguna, Spain.\newline
 \textit{Email address}: \href{mailto:victor.macia@ull.edu.es}{\texttt{victor.macia@ull.edu.es}}.}
\date{}

\begin{document}
\begin{abstract}
We give affirmative answers to both parts of Question~2.6 posed by Grosse-Erdmann in \cite{GE} concerning Wiman--Valiron inequalities in the unit disk.
For every unbounded analytic function in the disk, we establish the proposed iterated-logarithm inequalities outside exceptional sets of finite logarithmic measure.
The corresponding power estimate is a corollary.
Both conclusions follow from a variance bound for Khinchin families and a classical estimate for their largest atom.
The multiplicative constants in the main inequalities can be chosen absolute.
We also obtain a disk analogue of Rosenbloom's composition estimate.
The key step is a change of variable under which the variance of a rescaled Khinchin family equals the derivative of a nondecreasing auxiliary function.
A classical example shows that the leading logarithmic exponent $1/2$ cannot be decreased.
\end{abstract}
\maketitle

\section{Introduction and main result}\label{sec:introduction}

Let $f(z)=\sum_{k\geq0}a_kz^k$ be analytic in the unit disk $\D$.
For $t \in (0,1)$, write
\[
M(f,t)=\max_{|z|=t}|f(z)| \qquad \text{ and } \qquad \mu(f,t)=\max_{k\geq0}|a_k|t^k
\]
for the maximum modulus and the maximum term, respectively.  Cauchy's estimates for the coefficients yield $\mu(f,t)\leq M(f,t)$. 
Wiman--Valiron theory seeks inequalities in the opposite direction, up to additional factors and outside suitable exceptional sets of radii. In the unit disk, the maximum term may remain bounded even when the maximum modulus tends to infinity: the geometric series
$f(z)=\sum_{k\geq0}z^k=(1-z)^{-1}$ satisfies
$\mu(f,t)=1$ and $M(f,t)=(1-t)^{-1}\to\infty$, as $t\uparrow1$. This elementary example motivates the factor $(1-t)^{-1}$ appearing in the estimates below.
\smallskip

On the unit disk, the size of the exceptional set plays an essential role.
A measurable set $E\subset[0,1)$ has finite logarithmic measure if
\begin{equation}\label{eq:log_measure}
	\ell(E):=\int_E\frac{\dd t}{1-t}<\infty.
\end{equation}
Finite logarithmic measure implies logarithmic density zero, that is,
\[
\lim_{r\uparrow1}\frac{\ell(E\cap[0,r])}{\log\, (1/(1-r))}=0,
\]
but the converse fails (take, for instance, $E=\bigcup_{n\geq 1}[1-e^{-n^2},1-e^{-(n^2+1)}]$ and make the substitution $(1-t) = e^{-s}$).
Thus finite logarithmic measure is a strictly stronger requirement on the exceptional set.
It is the disk analogue of the condition $\int_{E\cap[1,\infty)}\dd r/r<\infty$ in the classical Wiman--Valiron theory for entire functions, see \cite[Section~1.1]{GE}.
\smallskip

Rosenbloom introduced in \cite{Rosenbloom} an elegant probabilistic approach to Wiman--Valiron inequalities based on Chebyshev's inequality. K\"ov\'ari adapted this approach to the unit disk in \cite{Kovari}.
For entire functions, a formulation of this approach in terms of Khinchin families is given in \cite[Section~1.5]{MaciaThesis}.
Building on Rosenbloom's approach and the estimates of Skaskiv and Kuryliak \cite{SK}, Grosse-Erdmann gave a unified treatment with a short proof in \cite{GE}.
Beyond these deterministic estimates, Wiman--Valiron inequalities also appear in the work of Agneessens and Grosse-Erdmann \cite{AGE} on the growth of random analytic functions and its applications to linear dynamics.
\smallskip

K\"ov\'ari's iterated-logarithm estimates hold outside sets of logarithmic density zero, see \cite[Thm~1.3 and eqn~(7)]{GE}.
Grosse-Erdmann strengthened this conclusion by obtaining an exceptional set of finite weighted measure, although this does not imply finite logarithmic measure \cite[Example~2.5]{GE}.
In \cite[Question~2.6]{GE}, Grosse-Erdmann asks whether the same bounds hold outside sets of finite logarithmic measure for every unbounded analytic function in the unit disk.
We answer both parts affirmatively: part~(a), concerning the iterated-logarithm estimate, and part~(b), concerning the power estimate.
We retain the unboundedness hypothesis to match the question. Bounded nonzero functions are treated in Section~\ref{sec:scope}.
\smallskip

To state the result, let $\log_1\, (y)=\log\, (y)$ and, recursively, $\log_{j+1}\, (y)=\log\, (\log_j\, (y))$ for $j\geq1$, whenever these expressions are defined.
Here and throughout, logarithms are natural, and empty products are understood to be $1$.

\begin{theo}\label{thm:disk_wiman_valiron}
	Let $f$ be an unbounded analytic function in $\D$, and put
	\[
	A(t)=\frac{\mu(f,t)}{1-t}.
	\]
Then there is an absolute constant $C>0$ (one may take $C=64\sqrt3$) such that for every integer $n\geq2$ and every $\delta>0$ there exists a measurable set $E\subset[0,1)$ with $\ell(E)<\infty$ such that
	\begin{equation}\label{eq:disk_iterated}
		M(f,t)\leq C \, A(t)\bigl(\log\, (A(t))\bigr)^{1/2} \left(\prod_{j=2}^{n-1}\log_j\, (A(t))\right) \bigl(\log_n\, (A(t))\bigr)^{1+\delta}, \qquad \text{ for any } t\in(0,1)\setminus E.
	\end{equation}
	The exceptional set $E$ may depend on $f,n,\delta$ and contains an initial interval on whose complement all displayed logarithms are defined and positive.
\end{theo}

\begin{coro}\label{cor:power}
	Under the same hypothesis, for every $\varepsilon>0$ there is a measurable set $E_\varepsilon\subset[0,1)$ with $\ell(E_\varepsilon)<\infty$ such that
	\begin{equation}\label{eq:disk_power}
		M(f,t)\leq C \,\frac{\mu(f,t)}{1-t} \left(\log\, \left(\frac{\mu(f,t)}{1-t}\right)\right)^{1/2+\varepsilon}, \qquad \text{ for any { }} t\in(0,1)\setminus E_\varepsilon.
	\end{equation}
	Here $C$ is the absolute constant in Theorem~\ref{thm:disk_wiman_valiron}.
\end{coro}
\smallskip

These results give affirmative answers to both parts of Grosse-Erdmann's Question~2.6 in \cite{GE}.
More precisely, Theorem~\ref{thm:disk_wiman_valiron} establishes exactly the inequality requested in part~(a), for every $n\geq2$ and $\delta>0$, with the exceptional set satisfying \eqref{eq:log_measure}.
Corollary~\ref{cor:power} answers part~(b) through \eqref{eq:disk_power}: taking $n=2$ and $\delta=1$ in Theorem~\ref{thm:disk_wiman_valiron}, the remaining iterated-logarithm factor can be absorbed into any prescribed positive power of $\log\, (A(t))$. Thus part~(b) follows from part~(a), while preserving finite logarithmic measure of the exceptional set.
\smallskip

The unconditional estimates of Sule\u\i manov and of Skaskiv and Kuryliak provide important precedents of Wiman--Valiron inequalities with an exceptional set of finite logarithmic measure.
They involve additional factors: $(1-t)^{-\delta}$ in Sule\u\i manov's power estimate~\cite{Suleimanov}, and $\bigl(\log\, (1/(1-t))\bigr)^{1/2+\delta}$ in the iterated-logarithm estimates following from Skaskiv and Kuryliak~\cite{SK}.
See \cite[Thm~1.4 and ineqs.~(8), (10), and~(11)]{GE}.
Our results remove these additional factors.
\smallskip

Special cases of the power estimate were already known under additional growth assumptions.
In particular, Skaskiv and Kuryliak~\cite[Thm~1, p.~110]{SK2010}, with $h(t)=(1-t)^{-1}$, obtain an estimate implying \eqref{eq:disk_power} when
\[
\liminf_{t\uparrow1} \frac{\log\, (\log\, (\Omega_f(t)))}{\log\, (1/(1-t))}>0, \qquad \text{ where } \qquad \Omega_f(t)=\sum_{k\geq0}|a_k|t^k.
\]
Our estimates hold without this additional growth assumption on the coefficientwise majorant.
For example, $g(z)=\exp\bigl((\log\, (1/(1-z)))^2\bigr)$ has nonnegative coefficients and satisfies $\log\, (\log\, (\Omega_g(t)))/\log\, (1/(1-t))\to0$, so it lies outside the hypothesis of \cite[Thm~1]{SK2010}.
\smallskip

Rosenbloom's method reduces the problem to bounding the variance of the Khinchin family outside a small exceptional set, and this bound is obtained from a growth lemma. In \cite[Thm~2.1]{GE}, Grosse-Erdmann applies a weighted form of that lemma with $h(t)=(1-t)^{-1}$, which introduces the additional factors in his estimates. We avoid the weight: after the change of variable $t=e^{-e^{-x}}$, the rescaled variance equals the derivative of a monotone auxiliary function. This allows us to apply the growth lemma without a weight and obtain the estimate without those additional factors.
\smallskip

For an unbounded analytic function in the unit disk, having nonnegative Taylor coefficients, let $m_f(t)$ and $\sigma_f^2(t)$ denote the mean and variance of its Khinchin family $(X_t)$, see Section~\ref{subsec:khinchin}.
Rosenbloom's estimate bounds $f(t)/\mu(f,t)$ by a constant times $1+\sigma_f(t)$, see \eqref{eq:maxterm_variance}.
Write
\[
F(s)=\log\, (f(e^s)),\quad \text{ for any } s<0, \quad \text{ and } \quad H(x)=F(-e^{-x}).
\]
The function $F$ satisfies $F'(s)=m_f(e^s)$ and $F''(s)=\sigma_f^2(e^s)$.
For $t=e^{-e^{-x}}$, exponential tilting of the Khinchin family acts by translation of $x$, see \eqref{eq:tilting_translation}.
The first two cumulants of the rescaled variable $(-\log\, (t))X_t$ are then $H'(x)$ and $(H+H')'(x)$.
Thus $U=H+H'=\log\, (f(t))+(-\log\, (t))m_f(t)$ satisfies
\begin{equation}\label{eq:key_identity_intro}
	U'(x)=e^{-2x}F''(-e^{-x}) =(-\log\, (t))^2\sigma_f^2(t)\geq0.
\end{equation}
Hence $U$ is nondecreasing.
This allows us to apply the growth lemma twice, to $H$ and $U$, outside a set of finite Lebesgue measure in $x$.
Since $\dd t/(1-t)\leq\dd x$, the corresponding set of radii has finite logarithmic measure, see \eqref{eq:measure_transfer_density} for further details.
\smallskip

Identity~\eqref{eq:key_identity_intro} lets us drop the weight: we apply the growth lemma to $H$ and $U$ directly, which gives the argument $H+\psi_1(H)$ in \eqref{eq:general_composition}.
The factor $(1-t)^{-1}$ in $A(t)$ arises when we pass from the rescaled standard deviation to $\sigma_f(t)$, using $e^x=1/(-\log\, (t))\leq1/(1-t)$.
Proposition~\ref{prop:boundary_calculus} and Theorem~\ref{thm:general_L} use this observation to obtain the bounds requested in Question~2.6 without the additional factors in the earlier estimates.
\smallskip

Corollary~\ref{cor:two_auxiliary} also gives a disk analogue of Rosenbloom's two-function estimate \cite[Thm~1.1]{GE}.
Its prefactor is $A(t) = \mu(f,t)/(1-t)$, and the composition is evaluated at $\log\, (A(t))$, without an additional factor $(1-t)^{-1}$ multiplying $\psi_1$ inside $\psi_2$.
Since $A(t)$ is unchanged when $f$ is replaced by its coefficientwise majorant, the same formulation applies to arbitrary complex coefficients.
This provides a common form for the disk inequalities discussed at the end of \cite[Section~1]{GE}.
\smallskip

This note is organized as follows.
Section~\ref{sec:preliminaries} develops the variance estimate.
Section~\ref{sec:general} proves the general bound, and Section~\ref{sec:iterated} deduces the main results and the two-function composition estimate.
Section~\ref{sec:scope} treats the leading exponent and extensions.

\section{Khinchin families and variance estimates}
\label{sec:preliminaries}

A power series with nonnegative coefficients and positive radius of convergence carries a natural family of probability distributions: for each $t \in (0,R)$, the terms $a_kt^k$, normalized by $f(t)$, form a probability distribution on the nonnegative integers.
Following an idea of Khinchin, Rosenbloom~\cite[p.~326]{Rosenbloom} introduced this probabilistic construction into complex analysis and derived the Wiman--Valiron inequality for entire functions from Chebyshev's inequality.
Related ideas appear earlier in Hayman's theory of admissible functions~\cite{Hayman56}, and later in the local central limit theorem of B\'aez-Duarte~\cite{BaezDuarte} for the partition function. Schumitzky extended Rosenbloom's probabilistic approach to entire functions of several complex variables~\cite{Schumitzky}. For a synthesis of these probabilistic arguments in the terminology of Khinchin families, see the author's doctoral thesis~\cite[Section~1.5]{MaciaThesis}. A systematic treatment in terms of Khinchin families, including the
class of strongly Gaussian power series, is given in \cite{CFFM1, CFFM2, CFFM3, MaciaGauss}.
\smallskip

Along with other methods, these probabilistic tools will be used to bound $M(f,t)$ in terms of $\mu(f,t)$.
We begin with the standard reduction to nonnegative coefficients, which we include for completeness.
\medskip

For $f(z)=\sum_{k\geq0}a_kz^k$ analytic in $\D$, its coefficientwise majorant $f_+(z)=\sum_{k\geq0}|a_k|z^k$ is analytic in $\D$, since the two series have the same radius of convergence.
For $t \in (0,1)$,
\[
    M(f,t)\leq f_+(t)=M(f_+,t),\qquad \mu(f_+,t)=\mu(f,t).
\]
Thus $A(t) = \mu(f,t)/(1-t)$ is unchanged, and an upper bound for $f_+(t)$ in terms of $A(t)$ gives the same bound for $M(f,t)$, with the same exceptional set.
\smallskip

If $f$ is unbounded, so is $f_+$, and its nonnegative coefficients give $f_+(t)\to\infty$, as $t\uparrow1$.
Moreover, $f_+$ then has infinitely many nonzero coefficients, since a polynomial is bounded on $\D$.
This reduction allows us to work with a nondegenerate Khinchin family while treating arbitrary complex coefficients.

\subsection{Khinchin families and maximum terms}\label{subsec:khinchin}

Let $f(z)=\sum_{k\geq0}a_kz^k$ be analytic in $\D$, with nonnegative coefficients and at least two nonzero coefficients.
We associate to $f$ its Khinchin family $(X_t)_{t \in (0,1)}$, consisting of one random variable $X_t$ for each $t \in (0,1)$, with values in the nonnegative integers and mass functions
\begin{equation}\label{eq:khinchin}
	\P(X_t=k)=\frac{a_kt^k}{f(t)},\qquad \text{ for any integer } k\geq0.
\end{equation}
The variables $X_t$ are not coupled, that is, $(X_t)_{t\in(0,1)}$ is a family indexed by $t$, not a stochastic process, and each $X_t$ is considered separately.
\smallskip

We write $m_f(t)=\E(X_t)$ and $\sigma_f^2(t)=\Var(X_t)$ for the mean and variance functions, which are given by
\[
m_f(t)=\frac{tf'(t)}{f(t)},\qquad \text{ and } \qquad \sigma_f^2(t)=t m_f'(t),\qquad \text{for any } t \in (0,1).
\]
The random variables $X_t$ take at least two different values, so the family is nondegenerate and $\sigma_f^2(t)>0$ for every $t \in (0,1)$.
\medskip

The largest mass in \eqref{eq:khinchin} is $\mu(f,t)/f(t)$.
Chebyshev's inequality concentrates mass in an interval whose length is controlled by the standard deviation.
Counting its integer points gives a lower bound for the largest atom, see also \cite[Lemmas~1.5.1 and~1.5.3]{MaciaThesis}.

\begin{lem}\label{lem:lattice}
Let $X$ be an integer-valued random variable with finite variance $\sigma^2$.
Then
\begin{equation}\label{eq:lattice_atom}
    \sup_{k\in\mathbb Z}\P(X=k) \geq\frac{3}{4(4\sigma+1)}.
\end{equation}
\end{lem}

\begin{proof}
If $\sigma=0$, a combination of Chebyshev's inequality and the continuity of the probability measure shows that $X$ is almost surely constant, and the bound holds. Assume that $\sigma \neq 0$. Then Chebyshev's inequality gives
\[
    \P\bigl(|X-\E X|\leq2\sigma\bigr)\geq\frac34.
\]
Put $I=[\E X-2\sigma,\E X+2\sigma]$.
This interval has length $4\sigma$, so it contains at most $\lfloor4\sigma\rfloor+1$ integers.
Therefore
\[
    \frac34\leq\P(X\in I) =\sum_{k\in I\cap\mathbb Z}\P(X=k) \leq(4\sigma+1)\sup_{k\in\mathbb Z}\P(X=k).
\]
Dividing by $4\sigma+1$ proves \eqref{eq:lattice_atom}.
\end{proof}

Lemma~\ref{lem:lattice}, applied to $X_t$, yields the following estimate:
\begin{equation}\label{eq:maxterm_variance}
  M(f,t) = f(t)\leq\frac43(4\sigma_f(t)+1)\mu(f,t) \leq\frac{16}{3}(1+\sigma_f(t))\mu(f,t).
\end{equation}
Estimate~\eqref{eq:maxterm_variance} follows the probabilistic approach of Rosenbloom~\cite{Rosenbloom}. See also the proof of Lemma~2.2 in~\cite{GE}. The linear dependence on $\sigma_f(t)$ in \eqref{eq:maxterm_variance} is optimal: for $f(z)=1/(1-z)$, the family \eqref{eq:khinchin} is geometric, with $\mu(f,t)=1$, $f(t)=(1-t)^{-1}$, and $\sigma_f(t)=\sqrt t/(1-t)$, so that $f(t)\sim\sigma_f(t)\,\mu(f,t)$, as the parameter $t\uparrow1$.
\smallskip

To control the variance, we use the fulcrum $F(s)=\log\, (f(e^s))$, which is defined for any $s<0$. In terms of $F$, the mean and variance functions satisfy
\begin{align}\label{eq:fulcrum_mean} 
    F'(s) = m_f(e^s), \quad \text{ and } \quad  F''(s) = \sigma_f^2(e^s).
\end{align}
Thus $F$ is strictly increasing and strictly convex.
If, in addition, $f$ is unbounded in $\D$, then $F(s) = \log(f(e^s))$ escapes to infinity, as $s\uparrow0$.

\subsection{An exceptional-set lemma}

The following growth lemma is a form of \cite[Lemma~2.3]{GE}.
We include its short proof for completeness.

\begin{lem}\label{lem:exceptional}
Let $v:[x_0,\infty)\to[y_0,\infty)$ be continuously differentiable and nondecreasing.
Let $\psi$ be continuous and positive on $[y_0,\infty)$, with
\begin{align}\label{eq: finite_integral}
    \int_{y_0}^{\infty}\frac{\dd y}{\psi(y)}<\infty.
\end{align}
Then the Borel set
\[
    B_v=\{x\geq x_0:v'(x)>\psi(v(x))\}
\]
has finite Lebesgue measure, and in fact
\begin{equation}\label{eq:exceptional_lemma}
    |B_v|\leq\int_{v(x_0)}^{\infty}\frac{\dd y}{\psi(y)}.
\end{equation}
\end{lem}

\begin{proof}
The defining strict inequality involves continuous functions, so $B_v$ is a Borel set.
For $b>x_0$, nonnegativity of $v'(x)$ combined with equation \eqref{eq: finite_integral} gives
\[
    |B_v\cap[x_0,b]| \leq\int_{x_0}^b\frac{v'(x)}{\psi(v(x))}\dd x =\int_{v(x_0)}^{v(b)}\frac{\dd y}{\psi(y)} \leq\int_{v(x_0)}^{\infty}\frac{\dd y}{\psi(y)}<\infty.
\]
Here we apply the change of variables $y=v(x)$. We conclude the proof by letting $b \to \infty$, and then using the continuity from below of the Lebesgue measure.
\end{proof}

\subsection{The variance estimate}

To control the variance in \eqref{eq:maxterm_variance} we take $t=e^{-e^{-x}}$, which turns finite Lebesgue measure in $x$ into finite logarithmic measure in $t$. Write $H(x)=F(-e^{-x})$ and $Y_x=e^{-x}X_t$. For real $u$ with $te^u<1$, equation \eqref{eq:khinchin} gives
\[
\E\,(e^{uX_t}) =\sum_{n=0}^{\infty}e^{un}\P(X_t=n) =\sum_{n=0}^{\infty}\frac{a_n(te^u)^n}{f(t)} =\frac{f(te^u)}{f(t)}.
\]
Reweighting the law of $X_t$ by $e^{uX_t}$ and normalizing by this constant is the exponential change of measure, see, for instance, \cite[Section~XV I.7]{FellerII}. It keeps $X_t$ within the Khinchin family and moves the parameter from $t$ to $te^u$, as the tilted probabilities show:
\[
\P(X_t=n)\cdot\frac{e^{un}}{\E\,(e^{uX_t})}=\frac{a_nt^n}{f(t)}\cdot\frac{e^{un}}{\E\,(e^{uX_t})}=\frac{a_n(te^u)^n}{f(te^u)}=\P(X_{te^u}=n).
\]

Since $te^u<1$ means $u<-\log\, (t)=e^{-x}$, the admissible range shrinks near the boundary. We therefore write $u=\lambda e^{-x}$ with $\lambda<1$, so that $uX_t=\lambda Y_x$, and record the first two cumulants of the rescaled variable $Y_x$.

\begin{lem}\label{lem:cumulants}
For every $t\in(0,1)$, put $x=-\log(-\log t)$. Then, for every $\lambda<1$, we have
	\begin{equation}\label{eq:tilting_translation}
		\log\, (\E\,(e^{\lambda Y_x}))
		=H\bigl(x-\log\, (1-\lambda)\bigr)-H(x).
	\end{equation}
	In particular,
	\begin{equation}\label{eq:rescaled_cumulants}
		\E\,(Y_x)=H'(x),\qquad \text{ and } \qquad \Var(Y_x)=H''(x)+H'(x)=(H+H')'(x).
	\end{equation}
\end{lem}
\begin{proof}
	By the change of measure above, $\E\,(e^{\lambda Y_x})=\E\,(e^{(\lambda e^{-x})X_t})=f(te^{\lambda e^{-x}})/f(t)$.
	Since $te^{\lambda e^{-x}}=e^{-e^{-x}(1-\lambda)}=e^{-e^{-(x-\log\, (1-\lambda))}}$, taking logarithms gives \eqref{eq:tilting_translation}.
	The right-hand side of \eqref{eq:tilting_translation} is finite for $\lambda<1$, so we may differentiate at $\lambda=0$.
	Writing $v(\lambda)=x-\log\, (1-\lambda)$, so that $v'(\lambda)=(1-\lambda)^{-1}$ and $v(0)=x$, the first two derivatives are
	\[
	\frac{\dd}{\dd\lambda}\,H(v(\lambda))
	=\frac{H'(v(\lambda))}{1-\lambda},
	\qquad \text{ and } \qquad
	\frac{\dd^2}{\dd\lambda^2}\,H(v(\lambda))
	=\frac{H''(v(\lambda))}{(1-\lambda)^2}+\frac{H'(v(\lambda))}{(1-\lambda)^2}.
	\]
	At $\lambda=0$ these equal $H'(x)$ and $H''(x)+H'(x)$.
	Since the first two derivatives at $\lambda=0$ of a cumulant generating function are the mean and the variance of $Y_x$, this gives \eqref{eq:rescaled_cumulants}.
\end{proof}

Identity \eqref{eq:rescaled_cumulants} is what makes the argument work. The mean $H'$ of $Y_x$ need not be monotone, but the variance $U'=(H+H')'$ is nonnegative, so $U=H+H'$ is nondecreasing. Bounding the variance thus amounts to bounding $U'$. The growth lemma, applied to $H$ and to $U$, controls both outside a set of finite measure, and the next proposition combines the two bounds into a single estimate for $U'=e^{-2x}F''(-e^{-x})$, valid for a general nondecreasing convex $F$.

\begin{propo}\label{prop:boundary_calculus}
Let $F\in C^2((s_0,0))$, where $s_0<0$, satisfy $F'\geq0$, $F''\geq0$, and $F(s)\to\infty$ as $s\uparrow0$.
Let $\psi_1,\psi_2$ be continuous, positive functions on $[y_0,\infty)$, with $\psi_2$ nondecreasing, such that
\[
    \int_{y_0}^{\infty}\frac{\dd y}{\psi_j(y)}<\infty, \qquad \text{ for any } j=1,2.
\]
Then there exist $x_0$ and a Borel set $B\subset[x_0,\infty)$ of finite Lebesgue measure such that
\begin{equation}\label{eq:general_composition}
    e^{-2x}F''(-e^{-x}) \leq\psi_2\left(F(-e^{-x})+\psi_1(F(-e^{-x}))\right), \qquad x\in[x_0,\infty)\setminus B.
\end{equation}
\end{propo}

\begin{proof}
Choose $x_0$ large enough that $-e^{-x_0}>s_0$ and $H(x_0)\geq y_0$. Recall that $H(x)=F(-e^{-x})$.
Put $U=H+H'$.
Direct differentiation gives
\begin{align}
    H'(x) &= e^{-x}F'(-e^{-x})\geq0,\label{eq:Hprime} \\
    H''(x) &= -e^{-x}F'(-e^{-x})+e^{-2x}F''(-e^{-x}),\notag \\
    U'(x) &= H'(x)+H''(x)=e^{-2x}F''(-e^{-x})\geq0. \label{eq:variance_cancellation}
\end{align}
This is the analytical form of the cumulant identity \eqref{eq:rescaled_cumulants} written in terms of the fulcrum.
Both $H$ and $U$ are continuously differentiable, nondecreasing, and unbounded, since $U\geq H\to\infty$.
Apply Lemma~\ref{lem:exceptional} to $H$ with $\psi_1$ and to $U$ with $\psi_2$.
The union $B$ of the two exceptional sets has finite measure.
For $x\in[x_0,\infty)\setminus B$, we have the inequalities
\[
    H'(x)\leq\psi_1(H(x)),\qquad \text{ and } \qquad U'(x)\leq\psi_2(U(x)).
\]
Using these inequalities we conclude that $U(x)\ = H(x)+H^{\prime}(x) \leq H(x)+\psi_1(H(x))$.
Monotonicity of $\psi_2$ and \eqref{eq:variance_cancellation} now prove \eqref{eq:general_composition}.
\end{proof}

\smallskip

Proposition~\ref{prop:boundary_calculus} gives an exceptional set in the variable $x$. To return to radii, recall that $t=e^{-e^{-x}}$. Since
\begin{equation}\label{eq:measure_transfer_density}
	\frac{\dd t}{1-t} =\frac{t(-\log\, (t))}{1-t}\dd x\leq\dd x \quad \text{ and } \quad \frac1{-\log\, (t)}\leq\frac1{1-t}, \quad \text{ for any } t \in (0,1),
\end{equation}
a finite-measure exceptional set in $x$ gives the desired exceptional set of radii.
The first inequality follows from $t(-\log\, (t))\leq1-t$ and the second from $1-t\leq-\log\, (t)$, both consequences of $\log\, (u)\leq u-1$ for $u>0$. The first turns a set of finite Lebesgue measure in $x$ into one of finite logarithmic measure in $t$. The second yields the factor $e^x\leq(1-t)^{-1}$ in the estimate for the standard deviation. Recall that we use the change of variable $t = e^{-e^{-x}}$.

\section{A general estimate and some auxiliary lemmas}\label{sec:general}

In this section we consider a class of auxiliary functions for which the self-composition in the variance bound of Proposition~\ref{prop:boundary_calculus} can be controlled. The definition uses two conditions. The integrability \eqref{eq:L_integral}, together with the monotonicity of $L$, ensures that the exceptional set has finite Lebesgue measure, see Lemma~\ref{lem:exceptional} above, and the composition bound \eqref{eq:L_composition} makes the cost of iterating $L$ only a constant factor, which is used to close the variance estimate in Lemma~\ref{lem:variance_bound_before_main}.

\begin{defin}\label{def:admissible}
A continuous nondecreasing function $L:[y_0,\infty)\to[1,\infty)$, with $y_0>1$, is called admissible if
\begin{equation}\label{eq:L_integral}
    \int_{y_0}^{\infty}\frac{\dd y}{yL(y)}<\infty,
\end{equation}
and there exists $K\geq1$ such that
\begin{equation}\label{eq:L_composition}
    L(2yL(y))\leq K L(y),\qquad \text{ for any } y\geq y_0.
\end{equation}
It is enough that these conditions hold after increasing $y_0$.
In particular, \eqref{eq:L_integral} forces $L(y)\to\infty$, as $y \rightarrow \infty$, so a constant $L$ is excluded from this class.
\end{defin}

Before looking at the general consequences of admissibility, we check that the functions used in Theorem~\ref{thm:disk_wiman_valiron} belong to this class.
\begin{lem}\label{lem:iterated_admissible}
	Fix an integer $n\geq2$ and $\delta>0$.
	There exists $y_0>1$ such that, on $[y_0,\infty)$, the function
\begin{equation}\label{eq:auxiliary_logarithms}
	L(y)=\left(\prod_{j=1}^{n-2}\log_j\, (y)\right) (\log_{n-1}\, (y))^{1+\delta}\,,
\end{equation}
is well defined, continuous, nondecreasing, and admissible in the sense of Definition~\ref{def:admissible}, with $K=2$.
\end{lem}
\begin{proof}
Fix $n\geq2$ and $\delta>0$.
For sufficiently large $y$, define
\[
L(y)=\left(\prod_{j=1}^{n-2}\log_j\, (y)\right) (\log_{n-1}\, (y))^{1+\delta}.
\]
Choose $y_0>1$ so large that every logarithm here is positive and $L(y)\geq1$ for $y\geq y_0$.
The function $L$ is continuous and increasing on this interval. We now prove that $L$ is admissible in the sense of Definition~\ref{def:admissible}.
\smallskip

The substitution $u=\log_{n-1}\, (y)$ gives
\begin{equation}\label{eq:reciprocal_integral}
	\int_{y_0}^{\infty}\frac{\dd y}{yL(y)} =\int_{\log_{n-1}\, (y_0)}^{\infty}\frac{\dd u}{u^{1+\delta}} <\infty.
\end{equation}
For $n=2$ this is the substitution $u=\log\, (y)$.
\smallskip

It remains to check \eqref{eq:L_composition}. For every $n\geq2$, taking logarithms in the definition of $L$ gives
\[
\log\, (L(y))=O(\log\, (\log\, (y)))=o(\log\, (y)), \qquad \text{ as } y \rightarrow \infty.
\]
Consequently
\[
\frac{\log\, (2yL(y))}{\log\, (y)}=1+\frac{\log 2+\log\, (L(y))}{\log\, (y)}\longrightarrow1, \qquad \text{ as } y\rightarrow\infty.
\]
Therefore $\log\, (2yL(y)) = \log\, (y)(1+o(1))$, as $y \rightarrow \infty$. Taking logarithms in the previous asymptotic relation, and iterating the process, we find that 
\[
\log_{j+1}\, (2yL(y)) =\log_{j+1}\, (y)+o(1), \qquad \text{ as } y \rightarrow \infty.
\]
Since $\log_{j+1}\, (y)\to\infty$, as $y \rightarrow \infty$, we obtain the following asymptotic relation: for any $j \geq 1$ we have
\[
{\log_j\, (2yL(y))}\sim {\log_j\, (y)}, \quad\text{ as } y \rightarrow \infty.
\]
Finally, using the finite product in \eqref{eq:auxiliary_logarithms}, we conclude that $L(2yL(y))\sim L(y)$, as $y \rightarrow \infty$. After increasing $y_0$, if necessary, \eqref{eq:L_composition} holds, for example with $K=2$. Thus $L$ is admissible.
\end{proof}

We now record the growth properties of an admissible $L$ written in terms of $W(y)=\sqrt y\,L(y)$.
\begin{lem}\label{lem:L_growth}
If $L$ is admissible and $W(y)=\sqrt y\,L(y)$, then $W$ is nondecreasing and
\begin{equation}\label{eq:W_bounds}
    W(3y)\leq\sqrt3 K^2 W(y),\qquad \text{ and } \qquad \log\, (W(y))=O(\log\, (y))=o(y),\qquad \text{ as } y\to\infty.
\end{equation}
\end{lem}

\begin{proof}
Since $L(y)\geq1$, for $y \geq 1$ large enough, monotonicity and \eqref{eq:L_composition} imply
\[
    L(2y)\leq L(2yL(y))\leq K L(y).
\]
Consequently $L(3y)\leq L(4y)\leq K^2 L(y)$, and hence $W(3y)\leq\sqrt3K^2W(y)$.
To obtain the second estimate, choose the integer $m\geq0$ such that $2^m y_0\leq y<2^{m+1}y_0$.
Then $L(y)\leq K^{m+1}L(y_0)$, so $\log\, (L(y))=O(\log\, (y))$.
Therefore $\log\, (W(y))=\tfrac12\log\, (y)+\log\, (L(y))=O(\log\, (y))=o(y)$.
\end{proof}

Admissibility now gives a bound for the variance of the Khinchin family in terms of the variable $x$.
\begin{lem}\label{lem:variance_bound_before_main} Let $f$ be an unbounded analytic function in $\D$ and let $L$ be admissible. Then there exist $x_0$ and a Borel exceptional set $B\subseteq[x_0,\infty)$ of finite Lebesgue measure such that
\begin{equation}\label{eq:variance_bound_inlemma}
	\sigma_{f_+}(t)\leq\sqrt{2K}\,e^x\,W(H(x)),\qquad \text{ for any }x\in[x_0,\infty)\setminus B.
\end{equation}
Recall that $H(x)=F(-e^{-x})$, $t=e^{-e^{-x}}$, and $W(y)=\sqrt y\,L(y)$. 
\end{lem}
\begin{proof}
By the reduction in Section~\ref{sec:preliminaries}, we may assume that $a_k\geq0$. Therefore $M(f,t)=f(t)$, and the associated Khinchin family is nondegenerate.
To estimate its variance, set
\[
\psi(y)=yL(y),\qquad \text{ and } \qquad W(y)=\sqrt y\,L(y).
\]
Apply Proposition~\ref{prop:boundary_calculus} to the fulcrum $F(s)=\log\, (f(e^s))$ with $\psi_1=\psi_2=\psi$.
Its hypotheses follow from \eqref{eq:fulcrum_mean} and \eqref{eq:L_integral}. The limit $F(s)\to\infty$ holds because $f$ is unbounded.
Let $B$ be the exceptional set given there.
At a fixed $x\in[x_0,\infty)\setminus B$, write $H=H(x)=F(-e^{-x})$ and $t=e^{-e^{-x}}$.
Since $L(H)\geq1$, we have $H \leq HL(H) = \psi(H)$, and therefore
\[
e^{-2x}\sigma_f^2(t) \leq \psi\bigl(H+\psi(H)\bigr) \leq\psi(2\psi(H)) = 2H L(H)L(2H L(H)) \leq2K H L(H)^2,
\]
which holds for any $x \in [x_0,\infty) \setminus B$. The last inequality follows from \eqref{eq:L_composition}. Observe that the left-hand side is $\Var(Y_x)$.
Thus
\begin{equation}\label{eq:variance_bound}
	\sigma_f(t)\leq\sqrt{2K}\,e^xW(H),\qquad \text{ for any }x\in[x_0,\infty)\setminus B.
\end{equation}
This concludes the proof. 
\end{proof}

We now combine this variance bound with Rosenbloom's maximum-term estimate, and transfer the exceptional set from $x$ to the set of radii, to obtain the general result.
\begin{theo}\label{thm:general_L}
Let $f$ be an unbounded analytic function in $\D$ and let $L$ be admissible.
Set $A(t)=\mu(f,t)/(1-t)$. Then, there exist $C=C(K)>0$ and a measurable set $E\subset[0,1)$ with $\ell(E)<\infty$ such that
\begin{equation}\label{eq:general_L_result}
    M(f,t)\leq C A(t)\sqrt{\log\, (A(t))}\,L(\log\, (A(t))), \qquad \text{ for any } t\in(0,1)\setminus E.
\end{equation}
The set $E$ may be chosen to contain an initial interval so that $\log\, (A(t))\geq y_0$ on its complement.
\end{theo}

\begin{proof} Fix $W(y) = \sqrt{y}L(y)$ and recall that $H(x) = F(-e^{-x}) = \log(f(t))$, where we use the relation $t = e^{-e^{-x}}$.
Using again the reduction in Section~\ref{sec:preliminaries}, we may assume that $a_k\geq0$. Consequently $M(f,t)=f(t)$, and the associated Khinchin family is nondegenerate. We apply Lemma \ref{lem:variance_bound_before_main} to obtain that there exists a Borel exceptional set $B \subseteq [x_0,\infty)$ of finite Lebesgue measure such that 
\begin{equation}\label{eq:variance_bound_ingeneral}
	\sigma_f(t)\leq\sqrt{2K}\,e^xW(H),\qquad \text{ for any }x\in[x_0,\infty)\setminus B.
\end{equation}
We now transfer the exceptional set to the interval of radii.
Increase $x_0$ if necessary, replace $B$ by $B\cap[x_0,\infty)$, and put $t_0=e^{-e^{-x_0}}$.
Define
\[
    E=[0,t_0]\cup\{e^{-e^{-x}}:x\in B\}.
\]
This is a Borel set because the change of variable is a homeomorphism onto its image.
By using equation \eqref{eq:measure_transfer_density}, we find that 
\begin{equation}\label{eq:measure_transfer}
	\ell(E)=\int_0^{t_0}\frac{\dd t}{1-t}+\int_B\frac{t(-\log\, (t))}{1-t}\,\dd x\leq-\log\, (1-t_0)+|B|<\infty.
\end{equation}
We may choose $x_0$ so large that $e^xW(H)\geq1$ and therefore Lemma \ref{lem:variance_bound_before_main} gives that
\begin{align}\label{eq: variance_and_W}
	1+\sigma_f(t)\leq 1+\sqrt{2K}\,e^xW(H)\leq (1+\sqrt{2K})\,e^xW(H),\qquad \text{ for any }  t\in(0,1)\setminus E.
\end{align}
Combining \eqref{eq:maxterm_variance}, \eqref{eq:variance_bound_ingeneral} and \eqref{eq: variance_and_W} with the inequality $e^x=1/(-\log\, (t))\leq1/(1-t)$, we find that, for any $t\in(0,1)\setminus E$, we have
\begin{equation}\label{eq:before_absorption}
  M(f,t)=f(t) \leq \frac{16}{3}(1+\sigma_f(t))\mu(f,t) \leq C_1 A(t)W(\log\, (f(t))).
\end{equation}
Here $C_1=\frac{16}{3}(1+\sqrt{2K})$ and $A(t) = \mu(f,t)/(1-t)$.

\smallskip
It remains to replace $\log\, (f(t))$ by $\log\, (A(t))$ in this bound.
For some fixed $k$ we have $a_k>0$, and hence $\mu(f,t)\geq a_kt^k$.
Therefore $A(t) = \mu(f,t)/(1-t) \to\infty$, as $t \uparrow 1$.
Write $h=\log\, (f(t))$ and $a=\log\, (A(t))$.
Both tend to infinity, as $t\uparrow1$.
By Lemma~\ref{lem:L_growth}, $\log\, (W(h))\leq h/2$ for all sufficiently large $h$.
Taking logarithms in \eqref{eq:before_absorption}, we obtain, for all sufficiently large $h$, the inequality
\[
    h\leq\log\, (C_1)+a+\log\, (W(h)) \leq\log\, (C_1)+a+\frac h2.
\]
Therefore we conclude that $h\leq2a+2\log\, (C_1)\leq3a$, for any $t\in(0,1)\setminus E$ sufficiently close to $1$.
Monotonicity and \eqref{eq:W_bounds} give $W(h)\leq W(3a)\leq\sqrt3 K^2W(a)$.
Substituting this into \eqref{eq:before_absorption} proves \eqref{eq:general_L_result} with $C(K)=\sqrt3K^2C_1$ for the nonnegative-coefficient function, after adding an initial interval to $E$ if needed.
Adding the initial interval keeps the exceptional set of finite logarithmic measure.
The reduction in Section~\ref{sec:preliminaries} gives the result for arbitrary complex coefficients.
\end{proof}

\section{Proof of the main results and a general composition estimate}\label{sec:iterated}

\begin{proof}[Proof of Theorem~\ref{thm:disk_wiman_valiron}] Fix $n\geq2$ and $\delta>0$.
	For sufficiently large $y$, define
\[
L(y)=\left(\prod_{j=1}^{n-2}\log_j\, (y)\right) (\log_{n-1}\, (y))^{1+\delta}.
\]
Applying Lemma~\ref{lem:iterated_admissible} we conclude that $L$ is admissible with $K=2$.
\smallskip

Now we apply Theorem~\ref{thm:general_L}. Since $K=2$, we may take $C=64\sqrt3$, independently of $f,n,\delta$.
With this choice of $L$, Theorem~\ref{thm:general_L} gives \eqref{eq:disk_iterated} with the factor $W(\log\, (A))=\sqrt{\log\, (A)}\,L(\log\, (A))$, and substituting $L$ yields
\[
\sqrt{\log\, (A)}\,L(\log\, (A)) =(\log\, (A))^{1/2} \left(\prod_{j=2}^{n-1}\log_j\, (A)\right) (\log_n\, (A))^{1+\delta}\,.
\]
This concludes the proof. 
\end{proof}

\begin{proof}[Proof of Corollary~\ref{cor:power}]\label{proof of cor:power}
Use Theorem~\ref{thm:disk_wiman_valiron} with $n=2$ and $\delta=1$.
For every $\varepsilon>0$, and for $A$ large enough, we have the inequality
\[
    (\log\, (\log\, (A)))^2\leq(\log\, (A))^\varepsilon\,.
\]
Since $A(t) = \mu(f,t)/(1-t) \to\infty$, as $t \uparrow 1$, the inequality \eqref{eq:disk_power} follows by enlarging the initial interval in the exceptional set.
\end{proof}

\begin{remark*}[A common exceptional set]
For each fixed $f$ as in Theorem~\ref{thm:disk_wiman_valiron}, the exceptional set can be chosen independently of $n$ and $\delta$, provided that each estimate is required only for radii sufficiently close to $1$, with a threshold that may depend on $n$ and $\delta$.
\smallskip

To see this, for each pair of integers $n\geq2$ and $m\geq1$, let $E_{n,m}$ be an exceptional set given by Theorem~\ref{thm:disk_wiman_valiron} with parameters $n$ and $\delta=1/m$. Since $\ell(E_{n,m})<\infty$ and $A(t)\to\infty$ as $t\uparrow1$, we can choose $r_{n,m}\in(0,1)$ sufficiently close to $1$ so that
\[
	\ell\bigl(E_{n,m}\cap[r_{n,m},1)\bigr)\leq 2^{-n-m} \qquad\text{and}\qquad \log_n\, (A(t))\geq1, \quad\text{for all }t\in[r_{n,m},1).
\]
Define
\[
	E=\bigcup_{n\geq2}\bigcup_{m\geq1}\bigl(E_{n,m}\cap[r_{n,m},1)\bigr).
\]
This set is measurable and has finite logarithmic measure, since
\[
   \ell(E)\leq\sum_{n\geq2}\sum_{m\geq1}2^{-n-m}=\frac12.
\]

Now fix any integer $n\geq2$ and any $\delta>0$, and choose an integer $m\geq1$ such that $1/m\leq\delta$. If $t\in[r_{n,m},1)\setminus E$, then $t\notin E_{n,m}$, so \eqref{eq:disk_iterated} holds with exponent $1+1/m$. Moreover, our choice of $r_{n,m}$ ensures that
\[
	\bigl(\log_n\, (A(t))\bigr)^{1+1/m}\leq\bigl(\log_n\, (A(t))\bigr)^{1+\delta}.
\]

Consequently, there is a single measurable set $E=E(f)$ of finite logarithmic measure such that, for every integer $n\geq2$ and every $\delta>0$, there exists a radius $r=r(f,n,\delta)\in(0,1)$ for which
\[
    M(f,t)\leq C\,A(t)\bigl(\log\, (A(t))\bigr)^{1/2}\left(\prod_{j=2}^{n-1}\log_j\, (A(t))\right)\bigl(\log_n\, (A(t))\bigr)^{1+\delta}, \qquad t\in[r,1)\setminus E.
\]
Here $C$ is the same absolute constant as in Theorem~\ref{thm:disk_wiman_valiron}. \hfill $\boxdot$
\end{remark*}

The normalization $A(t) = \mu(f,t)/(1-t)$ also permits a composition estimate with two independent auxiliary functions.
\begin{coro}\label{cor:two_auxiliary}
Let $f$ be an unbounded analytic function in $\D$, and let $\psi_1,\psi_2$ be positive, continuous, nondecreasing functions on $[y_0,\infty)$ such that
\[
    \int_{y_0}^{\infty}\frac{\dd y}{\psi_j(y)}<\infty, \qquad \text{ for any } j \in \{1,2\}.
\]
Put $A(t)=\mu(f,t)/(1-t)$ and $\Omega_f(t)=\sum_{k\geq0}|a_k|t^k$.
With the absolute constant $C=32/3$, there is a measurable set $E\subset[0,1)$ with $\ell(E)<\infty$ such that, for $t\in(0,1)\setminus E$, we have
\begin{align}
    M(f,t)\leq\Omega_f(t) &\leq C A(t)\sqrt{\psi_2\bigl(\psi_1(\tfrac12\log\, (\Omega_f(t)))\bigr)} \label{eq:two_auxiliary_majorant} \\
    &\leq C A(t)\sqrt{\psi_2\bigl(\psi_1(\log\, (A(t)))\bigr)}. \label{eq:two_auxiliary}
\end{align}
The exceptional set may depend on $f,\psi_1,\psi_2$ and includes an initial interval so that both compositions are defined.
\end{coro}

\begin{proof}
By the reduction in Section~\ref{sec:preliminaries}, it suffices to consider $g=f_+$. Then $g(t)=\Omega_f(t)\to\infty$, as $t\uparrow1$, while $A(t)$ is unchanged and $M(f,t)\leq g(t)$. Corollary~\ref{cor:power}, applied to $g$ with $\varepsilon=1/2$, gives $g(t)\leq C_0A(t)\log\, (A(t))$ outside a set $E_0$ of finite logarithmic measure. With $h=\log\, (g(t))$ and $a=\log\, (A(t))$, taking logarithms yields $h\leq a+\log\, (a)+\log\, (C_0)\leq2a$, for $t\in(0,1)\setminus E_0$ sufficiently close to $1$, since $a\to\infty$.
\smallskip

Monotonicity and integrability give
\[
   0\leq\frac{y}{2\psi_1(y)}\leq\int_{y/2}^{y}\frac{\dd u}{\psi_1(u)}\longrightarrow0,	\quad\text{ as }y\rightarrow\infty.
\]
Thus $\psi_1(y)/y\to\infty$, as $y\rightarrow\infty$, which also implies $y\leq\tfrac12\psi_1(y/2)$ for all sufficiently large $y$. To obtain a bound involving $\psi_2(\psi_1(h/2))$ and then use $h/2\leq a$, we introduce the rescaled auxiliary function $\widetilde\psi_1(y)=\psi_1(y/2)/2$ on a sufficiently large interval $[Y,\infty)$, with $Y\geq2\max\{y_0,1\}$, and restrict $\psi_2$ to the same interval. The function $\widetilde\psi_1$ is positive, continuous, nondecreasing, and satisfies  
\[
	\int_Y^{\infty}\frac{\dd y}{\widetilde\psi_1(y)}
	=4\int_{Y/2}^{\infty}\frac{\dd u}{\psi_1(u)}<\infty.
\]
\smallskip

Moreover, this choice ensures that
\[
	y+\widetilde\psi_1(y)\leq\psi_1(y/2),\qquad \text{ for any } y\geq Y.
\]

Apply Proposition~\ref{prop:boundary_calculus} to $F(s)=\log\, (g(e^s))$ with $\widetilde\psi_1$ and $\psi_2$. Writing $t=e^{-e^{-x}}$, we have $F(-e^{-x})=h$. By monotonicity of $\psi_2$ and $h+\widetilde\psi_1(h)\leq\psi_1(h/2)$, Proposition \ref{prop:boundary_calculus} yields
\[
	\sigma_g(t)\leq e^x\sqrt{\psi_2\bigl(\psi_1(h/2)\bigr)}
\]
outside a set $E_1$ of finite logarithmic measure, by \eqref{eq:measure_transfer_density}. Since $e^x\sqrt{\psi_2(\psi_1(h/2))}\to\infty$, as $t\uparrow1$, we have
\[
1+\sigma_g(t)\leq 1+e^x\sqrt{\psi_2\bigl(\psi_1(h/2)\bigr)}
\leq 2e^x\sqrt{\psi_2\bigl(\psi_1(h/2)\bigr)}\,,
\]
for $t\in(0,1)\setminus E_1$ sufficiently close to $1$. Combining this inequality with \eqref{eq:maxterm_variance} and $e^x\leq(1-t)^{-1}$ from \eqref{eq:measure_transfer_density}, we obtain
\[
	g(t)\leq C A(t)\sqrt{\psi_2\bigl(\psi_1(h/2)\bigr)} \leq C A(t)\sqrt{\psi_2\bigl(\psi_1(a)\bigr)},
\]
with $C=32/3$. The second inequality uses $h/2\leq a$ and monotonicity. These estimates hold for $t\in(0,1)\setminus(E_0\cup E_1)$ sufficiently close to $1$. Adding the remaining initial interval produces a set $E$ of finite logarithmic measure. Finally, $M(f,t)\leq g(t)$ proves \eqref{eq:two_auxiliary_majorant}--\eqref{eq:two_auxiliary}.
\end{proof}

\begin{remark*}[Comparison with {\cite[Thm~2.1]{GE}}]
Taking $h(t)=(1-t)^{-1}$ in \cite[Thm~2.1]{GE} gives
\[
	M(f,t)\leq C\,\frac{\mu(f,t)}{\sqrt{1-t}}\sqrt{\psi_2\left(\frac{1}{1-t}\,\psi_1\bigl(\log\, (M(f,t))\bigr)\right)}
\]
outside a set of finite logarithmic measure. In \eqref{eq:two_auxiliary_majorant} and \eqref{eq:two_auxiliary}, the factor $(1-t)^{-1}$ multiplying $\psi_1$ is absent, and the square root is multiplied by $C A(t)=C\mu(f,t)/(1-t)$.
\smallskip

For the logarithmic estimates \cite[ineqs.~(10) and~(11)]{GE}, Theorem~\ref{thm:disk_wiman_valiron} removes the additional factor $\bigl(\log\, (1/(1-t))\bigr)^{1/2+\delta}$ while retaining an exceptional set of finite logarithmic measure. \hfill $\boxdot$
\end{remark*}

\section{The leading exponent and extensions}\label{sec:scope}

\subsection{The leading logarithmic exponent}

The following example differs only in its constant term from the case $\varepsilon=1/2$ of Sule\u\i manov's example~\cite{Suleimanov}, as stated in \cite[Remark~1.8]{GE}.
We give a direct proof of an asymptotic equality with leading constant $2\sqrt\pi$, refining the lower bound in \cite[inequality~(12)]{GE} for this function.
\medskip

\begin{propo}\label{prop:sharpness}
	Let
	\[
	f(z)=\sum_{k\geq0}e^{\sqrt{k}}z^k.
	\]
	This is an unbounded analytic function in $\D$. With $A(t)=\mu(f,t)/(1-t)$,
	\begin{equation}\label{eq:sharp_count}
		M(f,t)\sim2\sqrt\pi\,A(t)\sqrt{\log\, (A(t))},\qquad \text{ as } t\uparrow1.
	\end{equation}
	In particular, no universal version of \eqref{eq:disk_power} with logarithmic exponent $\beta<1/2$ can hold outside sets of finite logarithmic measure.
\end{propo}

\begin{proof}
By the Cauchy--Hadamard formula, the radius of convergence of $f$ is $R = 1$, because $(e^{\sqrt{k}})^{1/k}\to1$, as $k\to\infty$.
The nonnegative coefficients also show that $M(f,t)=f(t)\geq(1-t)^{-1}$, so $f$ is unbounded.
Put $s=-\log\, (t)$ and $u_s=1/(4s^2)$, so that $t\in(0,1)$ corresponds to $s>0$, and $t\uparrow1$ corresponds to $s\downarrow0$. The inequalities below hold for every $s>0$, and the asymptotic relations are stated as $s\downarrow0$.
Completing the square gives
\begin{equation}\label{eq:complete_squares}
	\sqrt u-su=\frac1{4s} -s\left(\sqrt u-\frac1{2s}\right)^2,\qquad \text{ for any } u\geq0,
\end{equation}
so that $\mu(f,e^{-s}) = \max_{k \geq 0}\left(e^{\sqrt{k}-ks} \right) \leq e^{1/(4s)}$, for any $s>0$.
Choose an integer $k_s$ such that $|k_s-u_s|\leq1/2$. Such a choice is always possible, since for every real number there is an integer at most half a unit away. Using $\sqrt{u_s}=1/(2s)$, we obtain
\[
	\left|\sqrt{k_s}-\sqrt{u_s}\right| =\frac{|k_s-u_s|}{\sqrt{k_s}+\sqrt{u_s}}\leq\frac{1/2}{\sqrt{u_s}}= s,\qquad \text{ for any } s>0.
\]
The maximum term $\mu(f,e^{-s})$ is at least as large as the term with index $k_s$. Thus, by \eqref{eq:complete_squares},
\begin{align*}
	\log\, (\mu(f,e^{-s})) =\max_{k\geq0}(\sqrt{k}-sk) \geq\sqrt{k_s}-sk_s =\frac1{4s}-s\bigl(\sqrt{k_s}-\sqrt{u_s}\bigr)^2 \geq\frac1{4s}-s^3,
\end{align*}
where the last inequality follows from $|\sqrt{k_s}-\sqrt{u_s}|\leq s$. Combining this lower bound with the upper bound already obtained, we conclude that
\begin{equation*}
	\frac1{4s}-s^3\leq\log\, (\mu(f,e^{-s}))\leq\frac1{4s},\qquad \text{for any }s>0.
\end{equation*}
In particular,
\[
\left|\log\,(\mu(f,e^{-s}))-\frac1{4s}\right|\leq s^3, \quad \text{ for any } s > 0\,,
\]
and therefore
\[
\log\,(\mu(f,e^{-s}))=\frac1{4s}+O(s^3)=\frac1{4s}+o(1) \quad\text{ as }s\downarrow0\,,
\]
which finally implies that 
\begin{align}\label{eq:maximum_term_asymptotic}
A(e^{-s}) \sim \frac{1}{s}e^{1/(4s)}\,, \quad \text{ as } s \downarrow 0\,.
\end{align}
\smallskip

To estimate the full sum, put $q_s(u)=e^{\sqrt u-su}$. For any $s>0$, this function increases up to $u_s$ and decreases thereafter, with maximum $e^{1/(4s)}$. Comparing the sum with the integral on each unit interval, as in \cite[Section~13.1, p.~318]{HardyDivergent}, and using this monotonicity on either side of $u_s$, we obtain
\[
    \left|f(e^{-s})-\int_0^\infty q_s(u)\dd u\right| \leq\int_0^\infty|q_s'(u)|\dd u \leq2e^{1/(4s)},\qquad \text{ for any } s>0.
\]
The substitution $w=\sqrt{s}(\sqrt u-1/(2s))$ yields, for any $s>0$,
\[
	\int_0^\infty q_s(u)\dd u =\frac1s+s^{-3/2}e^{1/(4s)} \int_{-1/(2\sqrt s)}^\infty e^{-w^2}\dd w.
\]
As $s\downarrow0$, the last integral tends to $\int_{-\infty}^\infty e^{-w^2}\dd w=\sqrt\pi$, and $1/s$ is negligible against $s^{-3/2}e^{1/(4s)}$, so
\begin{align*}
  f(e^{-s})\sim\sqrt\pi\,s^{-3/2}e^{1/(4s)},\qquad \text{ as } s\downarrow0.
\end{align*}
Combining \eqref{eq:maximum_term_asymptotic} with $(1-e^{-s})\sim s$, as $s\downarrow0$, we find that
\[
  \log\, (A(e^{-s}))=\frac1{4s}+O(\log\, (1/s))\sim\frac1{4s}, \quad \text{ as } s \downarrow 0.
\]
Since $\sqrt{\log\, (A(e^{-s}))}\sim 1/(2\sqrt{s})$, as $\downarrow 0$, comparing the asymptotics of $f(e^{-s})$ and $A(e^{-s})$, we find \eqref{eq:sharp_count}.
\medskip
	
If $\beta<1/2$, then \eqref{eq:sharp_count} gives that $M(f,t)/\bigl(A(t)(\log\, (A(t)))^\beta\bigr)\to\infty$, as $t\uparrow1$.
Thus, for every fixed constant in a proposed upper bound with exponent $\beta$, that bound fails at every radius sufficiently close to $1$. Its exceptional set would have to contain an interval $(t_2,1)$, whose logarithmic measure is infinite.
\end{proof}

The factor $(1-t)^{-1}$ in our estimates is necessary, since the maximum term $\mu(f,t)$ and the logarithmic factor alone do not control the growth of $M(f,t)$: for $f(z)=1/(1-z)$ one has $\mu(f,t)=1$, so any bound of the form $M(f,t)\leq C\,\mu(f,t)(\log\,(A(t)))^\beta$ without it would force $M(f,t)$ to grow no faster than a power of $\log\,(1/(1-t))$, whereas $M(f,t)=(1-t)^{-1}\to\infty$.

\begin{remark}
Proposition~\ref{prop:sharpness} establishes sharpness of the leading term $A(t)\sqrt{\log\, (A(t))}$ as a whole: by \eqref{eq:sharp_count}, neither the prefactor $A(t)$ nor the exponent $1/2$ can be reduced.
It does not establish optimality of the iterated-logarithm factors in \eqref{eq:disk_iterated}, of the condition $\delta>0$, or of the size of the exceptional set.
In fact, \eqref{eq:sharp_count} shows that this example satisfies 
$$M(f,t)\leq C A(t)\sqrt{\log\, (A(t))}\,,$$
for every $t$ sufficiently close to $1$, without additional logarithmic factors.
This stronger endpoint corresponds to $L\equiv1$, whereas setting $\delta=0$ in \eqref{eq:disk_iterated} retains the iterated-logarithm factors.
Neither choice satisfies \eqref{eq:L_integral}. This limits the method without deciding whether either endpoint holds in general.
For entire functions, Hayman shows that $\delta$ cannot be set equal to zero in his inequality~(4.14), see \cite[pp.~333--334]{Hayman}. \hfill $\boxdot$
\end{remark}

\subsection{Other radii and bounded functions}

The restriction to the unit disk is only a normalization.
If $f$ is unbounded and analytic on $|z|<R$, where $0<R<\infty$, apply Theorem~\ref{thm:general_L} to $g(z)=f(Rz)$.
Then the same conclusion holds with
\[
    A_R(t)=\frac{\mu(f,t)}{1-t/R},\quad  \text{ for any } 0<t<R \qquad \text{ and } \qquad \int_E\frac{\dd t}{R-t}<\infty.
\]
Indeed $M(g,r)=M(f,Rr)$ and $\mu(g,r)=\mu(f,Rr)$, and the substitution $t=Rr$ preserves the displayed measure exactly.
\medskip

The upper bounds for $M(f,t)$ expressed solely in terms of $A(t)$ also hold for bounded nonzero analytic functions. Indeed, if $f$ is bounded and not identically zero, then $M(f,t)\leq C_f$ whereas $A(t) = \mu(f,t)/(1-t) \to\infty$, as $t \uparrow 1$, again, because $\mu(f,t)\geq|a_k|t^k$ for one nonzero coefficient. Thus the same conclusions hold after discarding an initial interval.
\medskip

\noindent\textbf{Acknowledgements.}
The author thanks Professor Jos\'e L. Fern\'andez for helpful comments
and suggestions, and Professor David Drasin for pointing out to them,
during the author's doctoral thesis, further work on Rosenbloom's
approach, such as~\cite{Schumitzky}.
\medskip

\noindent\textbf{Funding.}
V.~J.\ Maci\'a was partially supported by the Spanish Government through grants PID2023-148028NB-I00, PID2024-160326NA-I00 and PID2025-167726NB-I00.

\end{document}